\documentclass{amsart}
\usepackage{amssymb,amscd}
\usepackage{pstricks}
\usepackage{graphicx}

\numberwithin{equation}{section}

\newtheorem{theorem}{Theorem}[section]

\newtheorem{corollary}[theorem]{Corollary}
\newtheorem{lemma}[theorem]{Lemma}

\newtheorem{fact}[theorem]{Fact}
\newtheorem{example}[theorem]{Example}
\newtheorem{question}[theorem]{Question}

\DeclareMathOperator{\Lie}{Lie}

\newif\ifshowflag
\showflagtrue   
\newif\ifshowqstn
\showqstntrue   
\showqstnfalse 
\newif\ifshowinfo
\showinfotrue   
\showinfofalse 

\newcommand{\gex}{\gtrsim}  
\newcommand{\eqx}{\simeq}   

\renewcommand{\phi}{\varphi}

\DeclareMathOperator{\diam}{diam}

\DeclareMathOperator{\id}{id}

\newcommand{\mathfont}{\mathbb} 

\newcommand{\mfC}{{\mathfont C}}      
\newcommand{\mfN}{{\mathfont N}}      
\newcommand{\mfR}{{\mathfont R}}      
\newcommand{\mfS}{{\mathfont S}}      

\catcode`\@=11       

\def\rf#1{\@rf{#1}#1:;;}
\def\rfs#1{\@rfs{#1}#1:;;}
\def\rfm#1{\@rfF#1<>;;}

\def\@C{C}\def\@CC{CC}\def\@E{E}\def\@F{F}\def\@L{L}\def\@P{P}\def\@Q{Q}
\def\@R{R}\def\@S{S}\def\@T{T}\def\@TT{TT}\def\@X{X}\def\@s{s}\def\@f{f}

\def\@rf#1#2:#3;;{\def\@b{#2}
  \ifx\@b\@C Corollary~\ref{#1}\else%
  \ifx\@b\@CC Corollary~\ref{#1}\else%
  \ifx\@b\@E (\ref{#1})\else
  \ifx\@b\@F Fact~\ref{#1}\else%
  \ifx\@b\@L Lemma~\ref{#1}\else%
  \ifx\@b\@P Proposition~\ref{#1}\else%
  \ifx\@b\@Q Question~\ref{#1}\else%
  \ifx\@b\@R Remark~\ref{#1}\else%
  \ifx\@b\@S Section~\ref{#1}\else%
  \ifx\@b\@T Theorem~\ref{#1}\else%
  \ifx\@b\@TT Theorem~\ref{#1}\else%
  \ifx\@b\@X Example~\ref{#1}\else%
  \ifx\@b\@s \S\ref{#1}\else
  \ifx\@b\@f Figure~\ref{#1}\else%
  \ref{#1}\fi\fi\fi\fi\fi\fi\fi\fi\fi\fi\fi\fi\fi\fi}

\def\@rfs#1#2:#3;;{\def\@b{#2}
  \ifx\@b\@C Corollaries~\ref{#1}\else%
  \ifx\@b\@CC Corollaries~\ref{#1}\else%
  \ifx\@b\@F Facts~\ref{#1}\else%
  \ifx\@b\@L Lemmas~\ref{#1}\else%
  \ifx\@b\@P Propositions~\ref{#1}\else%
  \ifx\@b\@Q Questions~\ref{#1}\else%
  \ifx\@b\@R Remarks~\ref{#1}\else%
  \ifx\@b\@S Sections~\ref{#1}\else%
  \ifx\@b\@T Theorems~\ref{#1}\else%
  \ifx\@b\@TT Theorems~\ref{#1}\else%
  \ifx\@b\@X Examples~\ref{#1}\else%
  \ifx\@b\@s \S\ref{#1}\else
  \ifx\@b\@f Figures~\ref{#1}\else%
  \ref{#1}\fi\fi\fi\fi\fi\fi\fi\fi\fi\fi\fi\fi\fi}

\def\@rfF<#1>#2;;{\def\@c{#2}
  \@rfs{#1}#1:;;\ifx\@c\empty\else\@rfL:#2;;\fi}

\def\@rfL:#1<#2>#3;;{\def\@b{#2}\def\@c{#3}
  #1\ifx\@b\empty\else\ref{#2}\ifx\@c\empty\else\@rfL:#3;;\fi\fi}

\catcode`\@=12       

\def\vint_#1{\mathchoice
          {\mathop{\vrule width 6pt height 3 pt depth -2.5pt
                  \kern -8pt \intop}\nolimits_{\kern -4pt#1}}%
          {\mathop{\vrule width 5pt height 3 pt depth -2.6pt
                  \kern -6pt \intop}\nolimits_{#1}}%
          {\mathop{\vrule width 5pt height 3 pt depth -2.6pt
                  \kern -6pt \intop}\nolimits_{#1}}%
          {\mathop{\vrule width 5pt height 3 pt depth -2.6pt
                   \kern -6pt \intop}\nolimits_{#1}}}
\newcommand{\intavg}{\vint}        

\begin{document} 

\title[Bilipschitz Homogeneity]{Transitive bi-Lipschitz group actions and bi-Lipschitz parameterizations}
\date{\today}

\author{David M. Freeman}
\address{%
\footnotesize
\begin{tabular}{lr}
\textsf{University of Cincinnati Blue Ash College}\\
\textsf{9555 Plainfield Road, Cincinnati, Ohio 45236}
\end{tabular}
\normalsize}
\email{david.freeman@uc.edu}


\begin{abstract} 
We prove that Ahlfors $2$-regular quasisymmetric images of $\mfR^2$ are bi-Lipschitz images of $\mfR^2$ if and only if they are uniformly bi-Lipschitz homogeneous with respect to a group. We also prove that certain geodesic spaces are bi-Lipschitz images of Carnot groups if they are inversion invariant bi-Lipschitz homogeneous with respect to a group. 
\end{abstract}

\maketitle

\section{Introduction}\label{S:intro}

The problem of finding bi-Lipschitz parameterizations (as in \cite{Semmes-params}) and its connections with the so-called \textit{Quasiconformal Jacobian Problem }(as in \cite{BHS-QC}) has attracted considerable interest. We begin the present paper by examining a few ways the concepts of bi-Lipschitz homogeneity and quasihomogeneous parameterizations relate to these problems. Our initial results concern Ahlfors $2$-regular metric spaces $X$ that are quasisymmetrically homeomorphic to $\mfR^2$. We show that if the space admits a transitive uniformly bi-Lipschitz group action, then the space can be parametrized by a bi-Lipschitz map $f:\mfR^2\to X$. We also show that in higher dimensions the $n$-regularity of $X$ implies that any quasihomogeneous homeomorphism $f:\mfR^n\to X$ must be bi-Lipschitz. Therefore, if one wants to construct a bi-Lipschitz parameterization of an $n$-regular space, it is sufficient to construct a quasihomogeneous parameterization. Analogous results (in the two-dimensional case) pertaining to the Quasiconformal Jacobian Problem are also noted. 

In the remainder of the paper we no longer restrict ourselves to parameterizations of the form $f:\mfR^n\to X$. Instead, we find bi-Lipschitz parameterizations of the form $f:\mathbb{G}\to X$, where $\mathbb{G}$ denotes a Carnot group and $X$ denotes a certain class of geodesic metric spaces. To obtain the desired bi-Lipschitz parameterizations in this setting we focus on the property of inversion invariant bi-Lipschitz homogeneity as studied in \cite{Freeman-iiblh}. 

The main results are listed in \rf{S:main}, followed by an explanation of our notational conventions in \rf{S:prelims}. We examine relevant properties of strong $A_\infty$ weights in \rf{S:weights}, and the relationship between quasihomogeneous and bi-Lipschitz maps in \rf{S:QH}. We are then ready to prove \rf{T:R2} and \rf{T:QHBL} in \rf{S:BL} and \rf{S:QH}, respectively. We study consequences of inversion invariant bi-Lipschitz homogeneity in \rf{S:inversion}, and prove the remaining results in \rf{S:BLC} and \rf{S:remains}.

\section{Main Results}\label{S:main}

\begin{theorem}\label{T:R2}
Suppose $X$ is an Ahlfors 2-regular, linearly locally connected proper metric space that is homeomorphic to $\mfR^2$. The following are equivalent.
\begin{itemize}
  \item[$(a)$]{$X$ is uniformly bi-Lipschitz homogeneous with respect to a group}
  \item[$(b)$]{There exists a bi-Lipschitz homeomorphism $f:\mfR^2\to X$.}
  \item[$(c)$]{There exists a quasihomogeneous homeomorphism $f:\mfR^2\to X$.} 
\end{itemize}
\end{theorem}

We include item $(c)$ in the above theorem due to the prominent appearance of quasihomogeneous maps in the theory of bi-Lipschitz homogeneous Jordan curves (see \cite{HM-blh}, \cite{GH-blh}, \cite{Bishop-blh}). In fact the relationship between $(b)$ and $(c)$ can be strengthened and generalized as follows. 

\begin{theorem}\label{T:QHBL}
Suppose $X$ is an Ahlfors $n$-regular metric space. A homeomorphism $f:\mfR^n\to X$ is quasihomogeneous if and only if it is bi-Lipschitz. 
\end{theorem}

The proof of \rf{T:R2} relies on the characterization of quasiconformal groups acting on $\mfR^2$ given by Tukia's work in \cite{Tukia-conjugate} (see also \cite[Theorem 12.1]{Gehring-topics}), the quasisymmetric uniformization theory of Bonk, Kleiner, and Wildrick (see \cite{BK-sphere},\cite{Wildrick-planes}), along with work of David and Semmes from \cite{Semmes-strong} and \cite{DS-infinity}. 

Due to considerations outlined in \cite{BHS-QC}, \rf{T:R2} is equivalent to the following result. See \rf{S:weights} for further discussion of strong $A_\infty$ weights and an explanation of \rf{C:QC}.

\begin{corollary}\label{C:QC}
Suppose $\omega$ is a strong $A_\infty$ weight on $\mfR^2$. The following two items are equivalent:
\begin{itemize}
  \item[$(a)$]{There exists a quasiconformal homeomorphism $f:\mfR^2\to\mfR^2$ and a constant $1\leq C<+\infty$ such that for almost every $x\in \mfR^2$ we have $C^{-1}Jf(x)\leq \omega(x)\leq C\,Jf(x)$.}
  \item[$(b)$]{The space $(\mfR^2,d_\mu)$ corresponding to $\omega$ is uniformly bi-Lipschitz homogeneous with respect to a group.}
\end{itemize}
\end{corollary}

Departing from the 2-dimensional setting, we examine the condition of inversion invariant bi-Lipschitz homogeneity as studied in \cite{Freeman-iiblh}. In the setting of certain geodesic metric spaces we obtain the following result. 

\begin{theorem}\label{T:Carnot}
Suppose $X$ is an unbounded, proper, doubling, and geodesic metric space. If $X$ is inversion invariant bi-Lipschitz homogeneous with respect to a group, then $X$ is bi-Lipschitz equivalent to a Carnot group $\mathbb{G}$ equipped with a Carnot-Caratheodory distance associated to the horizontal layer of $\Lie(\mathbb{G})$.
\end{theorem}

We focus attention on inversion invariant bi-Lipschitz homogeneity primarily due to its role in the $1$-dimensional case. In particular, we have the following result from \cite{Freeman-iiblh}. Recall that a quasi-line (or quasi-circle) is a quasisymmetric image of the real line (or unit circle).

\begin{fact}\label{F:Q-regular}
Suppose $\Gamma$ is a proper and doubling metric space homeomorphic to the real line (or unit circle). $\Gamma$ is an Ahlfors $Q$-regular quasi-line (or quasi-circle) if and only if $\Gamma$ is inversion invariant bi-Lipschitz homogeneous.  
\end{fact}

Thus we obtain a characterization of inversion invariant bi-Lipschitz homogeneity for certain $1$-dimensional metric spaces. We view \rf{T:Carnot} as a step towards a characterization of this property in higher dimensions.

Regarding the assumption of the doubling property in \rf{T:Carnot}, we point out that Le Donne has shown that locally bi-Lipschitz homogeneous geodesic metric surfaces are locally doubling (\cite[Proposition 3.7]{LeDonne-blh}). 

\medskip
One might ask whether or not all Carnot groups are inversion invariant bi-Lipschitz homogeneous, but the answer to this question is not immediately clear. For example, the usual Heisenberg group satisfies this property via Kor\'anyi inversion. However, by results from \cite{CDKR-Heisenberg} not all Carnot groups possess such an inversion.

\medskip
Recall that when a Carnot group is of step greater than $1$, Hausdorff dimension strictly exceeds topological dimension. This observation provides the following corollary, which offers one way of characterizing $\mfR^n$ up to bi-Lipschitz homeomorphisms.

\begin{corollary}\label{C:Rn}
Suppose $X$ is doubling and proper metric space of Hausdorff dimension $n$ that is homeomorphic to $\mfR^n$. Then $X$ is inversion invariant bi-Lipschitz homogeneous with respect to a group if and only if $X$ is bi-Lipschitz equivalent to $\mfR^n$.
\end{corollary}

The proof of \rf{T:Carnot} utilizes arguments from \cite{LeDonne-tangents} and relies on the following theorem, which may be of interest in its own right. A space is \textit{quasi-self-similar} provided that, up to bi-Lipschitz equivalence, the space is invariant  under rescalings (see \rf{S:inversion} for a precise definition). 

\begin{theorem}\label{T:QSS}
Suppose $X$ is a proper, connected, and inversion invariant bi-Lipschitz homogeneous space. Then $X$ is quasi-self-similar. 
\end{theorem}

In fact, \rf{T:QSS} is quantitative in the sense that the quasi-self-similarity constant depends only on the inversion invariant bi-Lipschitz homogeneity constant.  We emphasize that we are not assuming homogeneity with respect to a group in this theorem.

One might ask if quasi-self-similarity along with bi-Lipschitz homogeneity implies inversion invariant bi-Lipschitz homogeneity. In general, the answer is `no.' A modified version of the surface known as \textit{Rickman's Rug} is a quasi-self-similar and bi-Lipschitz homogeneous surface in $\mfR^3$, but its inversion is not bi-Lipschitz homogeneous (see \rf{X:rug}). 

\medskip
We conclude this section by emphasizing that the proofs of \rf{T:R2} and \rf{T:Carnot} heavily rely on the group structure contained in their assumptions. It would be interesting to know if different proofs exist that do not use this assumption. For example, the following question flows naturally from \rf{T:R2} above.

\begin{question}\label{Q:R2}
Suppose $X$ is a linearly locally connected, Ahlfors 2-regular, and proper metric space that is homeomorphic to $\mfR^2$. If $X$ is uniformly bi-Lipschitz homogeneous, is $X$ bi-Lipschitz equivalent to $\mfR^2$?
\end{question}

When considering \rf{Q:R2}, it may be helpful to examine spaces such as the surface $S$ constructed by Bishop in \cite{Bishop-surface}. The surface $S$ is the image of $\mfR^2\subset\mfR^3$ under a quasiconformal (and hence quasisymmetric) self homeomorphism of $\mfR^3$, so it is linearly locally connected. It is also Ahlfors $2$-regular. However, $S$ is not bi-Lipschitz equivalent to $\mfR^2$. It is relevant to our discussion to note that $S$ is \underline{not} bi-Lipschitz homogeneous (basically for the same reason that it is not bi-Lipschitz equivalent to $\mfR^2$). 

When considering a higher dimensional version of \rf{Q:R2}, one might examine spaces such as those constructed by Semmes in \cite{Semmes-params}. In particular, Semmes constructed an Ahlfors $3$-regular metric space $E$ that is the image of $\mfR^3\subset\mfR^4$ under a global quasiconformal self homeomorphism of $\mfR^4$, yet is not bi-Lipschitz equivalent to $\mfR^3$ (due to considerations related to Hausdorff measure and topology). Again one can show that that this space is \underline{not} bi-Lipschitz homogeneous.

For a compact example, consider the \textit{polyhedral Edwards sphere} $\Sigma^2H^3$. This is the double suspension of a non-simply connected homology 3-sphere $H^3$. We refer to \cite[Question 12]{HS-questions} for a brief discussion of this space in a context relevant to our considerations. It was pointed out by Sullivan that $\Sigma^2H^3$ is not bi-Lipschitz equivalent to the standard $5$-sphere $\mfS^5$, even though $\Sigma^2H^3$ is homeomorphic to $\mfS^5$ (by celebrated results of Edwards and Cannon). It is not very hard to show that $\Sigma^2H^3$ is \underline{not} bi-Lipschitz homogeneous. Indeed, points on the \textit{suspension circle} in $\Sigma^2H^3$ cannot be mapped off of the suspension circle by a bi-Lipschitz self homeomorphism of $\Sigma^2H^3$ (again for reasons related to Hausdorff measure and topology). 

In light of these various (non-)examples, it seems difficult to find a `nice' (i.e.\,homeomorphic to $\mfR^n$, Ahlfors $n$-regular, quasiconvex,...) space that is \underline{not} bi-Lipschitz equivalent to $\mfR^n$ but \underline{is} bi-Lipschitz homogeneous. 

\medskip
The author would like to thank Stephen Semmes for a helpful explanation of the polyhedral Edwards sphere, and Enrico Le Donne for an abundance of helpful feedback during the preparation of this paper. The author is also indebted to the anonymous referee for their helpful comments.

\medskip
\textit{Added in proof:} The equivalence between $(a)$ and $(b)$ in \rf{T:R2} can be established by means other than those employed in the present work and is true in greater generality. In particular, we point out the following fact:

\begin{fact}
Let $X$ be a proper, doubling, quasiconvex space homeomorphic to $\mfR^n$ that has Hausdorff dimension $n$. The space $X$ admits a transitive uniformly bi-Lipschitz group action if and only if $X$ is bi-Lipschitz equivalent to $\mfR^n$.
\end{fact}

This fact can be verified by the use of the solution to Hilbert's Fifth Problem and work of Berestovskii pertaining to isometrically homogeneous geodesic spaces. One uses this theory to show that $X$ is bi-Lipschitz equivalent to an isometrically homogeneous sub-Riemannian manifold (the argument is referenced in \cite[pg. 569]{LeDonne-geodesic} and \cite[pg. 783]{LeDonne-blh}). Due to the assumption that the Hausdorff dimension is $n$, the metric must be Riemannian (see the proof of \rf{C:Rn}). The only simply connected manifold of this type is $\mfR^n$ itself. We prove \rf{T:R2} by very different means. We do not pass to isometric homogeneity in order to utilize Lie group theory. Instead, we work within the theory of quasisymmetric mappings, two-dimensional quasiconformal groups, and $A_\infty$ weights. 

Of course, \rf{C:Rn} can also be proved by the same observations, thus rendering the assumption of inversion invariance unnecessary. We include this assumption in order to demonstrate some consequences of \rf{T:QSS}. 

\section{Preliminaries}\label{S:prelims}

Given two numbers $A$ and $B$, we write $A\eqx_C B$ to indicate that $C^{-1}A\leq B\leq CA$, where $C$ is  independent of $A$ and $B$. When the quantity $C$ is understood or need not be specified, we simply write $A\eqx B$. 

Given a metric space $(X,d)$, for $r>0$ and $x\in X$ we write
\[B_d(x;r):=\{y\in X:d(x,y)<r\}\]
to denote an open ball. For $0<r<s$, spheres and annuli are written as
\[S_d(x;r):=\{y\in X:d(x,y)=r\}\]
\[A_d(x;r,s):=\{y\in X: r<d(x,y)<s\}.\]
We omit subscripts in the above notation when no confusion is possible. 

Given $1\leq L<+\infty$, an embedding $f:X\to Y$ is \textit{$L$-bi-Lipschitz} provided that for all points $x,y\in X$ we have
\[L^{-1}d_X(x,y)\leq d_Y(f(x),f(y))\leq L\,d_X(x,y).\]
Two spaces $X$ and $Y$ are \textit{$L$-bi-Lipschitz equivalent} if there exists an $L$-bi-Lipschitz homeomorphism between the two spaces. A space $X$ is \textit{bi-Lipschitz homogeneous} if there exists a collection $\mathcal{F}$ of bi-Lipschitz self-homeomorphisms of $X$ such that for every pair $x,y\in X$ there exists $f\in \mathcal{F}$ with $f(x)=y$. When we can take every map in $\mathcal{F}$ to be $L$-bi-Lipschitz we say that $X$ is $L$-bi-Lipschitz homogeneous, or \textit{uniformly bi-Lipschitz homogeneous} when the particular constant is not important. When the collection $\mathcal{F}$ forms a group, we say that $X$ is bi-Lipschitz homogeneous \textit{with respect to a group}. 


\medskip
A metric space $(X,d)$ is \textit{quasiconvex} provided that there exists a constant $1\leq C<+\infty$ such that any two points $x,y\in X$ can be joined by a path $\gamma$ whose length is no greater than $C\,d(x,y)$. Note that a quasiconvex space is bi-Lipschitz equivalent to a length space. When the space is also complete and \textit{proper} (closed balls are compact) it is bi-Lipschitz equivalent to a geodesic space: there exist length minimizing paths between any two points.

For $Q>0$, a space $(X,d)$ is \textit{Ahlfors $Q$-regular} provided that any ball $B(x;r)\subset X$ has Hausdorff measure comparable to $r^Q$. More precisely, writing $\mathcal{H}^Q$ to denote the usual $Q$-dimensional Hausdorff measure, we have $\mathcal{H}^Q(B(x;r))\eqx r^Q$, where the comparability is independent of $x$ and $r$. 

For $\lambda>1$, we say that $X$ is \textit{$\lambda$-linearly locally connected} provided that for all $a\in X$ and $0<r<\diam(X)$ we have (see, for example, \cite[p. 787]{Wildrick-planes})\\
\begin{enumerate}
  \item{For each pair of distinct points $\{x,y\}\subset B(a;r)$ there exists a continuum $E\subset B(a;\lambda r)$ containing $\{x,y\}$.}
  \item{For each pair of distinct points $\{x,y\}\subset X\setminus B(a;r)$ there exists a continuum $E\subset X\setminus B(a;r/\lambda)$ containing $\{x,y\}$.}\\
\end{enumerate}

A metric space $(X,d)$ is \textit{metric doubling} provided that there exists some $1\leq N<+\infty$ such that every ball $B(x;r)$ can be covered by at most $N$ balls of radius $r/2$. Note that $N$ is independent of $x$ and $r$.

\section{Strong $A_\infty$ Weights}\label{S:weights}

In order to prove \rf{T:R2} we make use of the theory that is outlined in \cite{BHS-QC}. In particular, we make use of \textit{metric doubling measures} arising from \textit{strong $A_\infty$ weights}. For a more thorough treatment we refer the reader to \cite{DS-infinity} and \cite{Semmes-strong}. Here we provide a few necessary definitions.

A \textit{weight} $\omega$ is a nonnegative locally integrable function on $\mfR^n$. A weight $\omega$ is called an \textit{$A_\infty$ weight} provided that there exist contants $\varepsilon>0$ and $1\leq C<+\infty$ such that 
\[\left(\intavg_B\omega^{1+\varepsilon}(x)d\mathcal{H}^n(x)\right)^{1/(1+\varepsilon)}\leq C\,\intavg_B\omega(x)d\mathcal{H}^n(x).\]
Here $B$ is any ball in $\mfR^n$ and the barred integral sign indicates the integral average as follows:
\[\intavg_E f(x)d\mathcal{H}^n(x):=\frac{1}{\mathcal{H}^n(E)}\int_E f(x)d\mathcal{H}^n(x).\]

A function $\delta:X\times X\to[0,+\infty)$ is a \textit{quasidistance} (sometimes called a \textit{quasimetric}) provided that it is symmetric, vanishes precisely on the diagonal of $X\times X$, and satisfies the following weak version of the triangle inequality: There exists a constant $1\leq C<+\infty$ such that for all $x,y,z\in X$, 
\[\delta(x,z)\leq C(\delta(x,y)+\delta(y,z)).\]

Suppose that $\mu$ is a \textit{doubling measure} on $\mfR^n$. That is, $\mu$ is non-trivial and there exists a constant $1\leq D<+\infty$ such that for every $B(x;r)\subset\mfR^n$ we have $\mu(B(x;2r))\leq D\mu(B(x;r))$. Given such a measure, we define a quasidistance 
\[\delta_\mu(x,y):=\mu(B_{xy})^{1/n}.\]
Here $B_{xy}=B(x;|x-y|)\cup B(y;|x-y|)$.

We say that $\mu$ is a \textit{metric doubling measure} provided that there exists a distance function $d$ on $\mfR^n$ and a constant $1\leq C<+\infty$ such that for all $x,y\in \mfR^n$, 
\[C^{-1}d(x,y)\leq \delta_\mu(x,y)\leq C\,d(x,y).\]

By results from \cite{DS-infinity} (see also \cite[Proposition 3.4]{Semmes-strong}) a metric doubling measure $\mu$ has an $A_\infty$-density $\omega$ such that $d\mu(x)=\omega(x)d\mathcal{H}^n(x)$. We refer to densities $\omega$ which arise in this way as \textit{strong $A_\infty$ weights}. 

The preceding paragraphs allow us to conclude that for each strong $A_\infty$ weight $\omega$ there exists a corresponding metric space $(\mfR^n,d_\mu)$, where $d_\mu$ is comparable to the quasidistance defined via the measure $\mu(E):=\int_E\omega(x)d\mathcal{H}^n(x)$. 

In the introduction of \cite{BHS-QC} it is pointed out that a bijection $f:\mfR^2\to(\mfR^2,d_\mu)$ is bi-Lipschitz if and only if $f:\mfR^2\to\mfR^2$ is quasiconformal and $Jf$ is almost everywhere comparable to $\omega$. This idea immediately yields $(a)\Rightarrow(b)$ in \rf{C:QC}. The same idea also yields $(b)\Rightarrow(a)$. To see this, suppose that one is given a strong $A_\infty$ weight $\omega$ on $\mfR^2$ for which $(\mfR^2,d_\mu)$ is uniformly bi-Lipschitz homogeneous with respect to a group. By observations made in \cite[pg. 342-343]{Semmes-params} one can verify that $(\mfR^2,d_\mu)$ is quasisymmetrically equivalent to $\mfR^2$ and Ahlfors $2$-regular. Therefore, by \rf{T:R2}, there exists a bi-Lipschitz homeomorphism $f:\mfR^2\to(\mfR^2,d_\mu)$. This allows us to conclude that $\omega$ is almost everywhere comparable to $Jf$. 

With these thoughts in mind we record the following lemma.

\begin{lemma}\label{L:dbling-bl}
Suppose $\omega$ is a strong $A_\infty$ weight on $\mfR^2$ and $d_\mu$ is a distance comparable to the quasidistance defined via $\omega$. A bijection $f:(\mfR^2,d_\mu)\to(\mfR^2,d_\mu)$ is a bi-Lipschitz homeomorphism if and only if $f:\mfR^2\to\mfR^2$ is a quasiconformal homeomorphism and for almost every $x\in \mfR^2$ we have $\omega(x)\eqx\omega(f(x))Jf(x)$. The relevant constants depend only on each other.
\end{lemma}

\begin{proof}
Assume that $f:(\mfR^2,d_\mu)\to(\mfR^2,d_\mu)$ is a bi-Lipschitz homeomorphism. So for $x,y\in \mfR^2$, $d_\mu(f(x),f(y))\eqx d_\mu(x,y)$. Using the doubling property of $\mu$,
\[d_\mu(x,y)\eqx \delta_\mu(x,y)=\mu(B_{xy})^{1/2}\eqx\mu(B(x;|x-y|))^{1/2}.\]
It follows from the bi-Lipschitz property of $f$ that
\begin{equation}\label{E:ball-measure} 
\mu(B(x;|x-y|))\eqx\mu(B(f(x);|f(x)-f(y)|)).
\end{equation}
Define $r:=|x-y|$ and set
\[s:=\min_{z\in S(x;r)}|f(x)-f(z)|\qquad t:=\max_{z\in S(x;r)}|f(x)-f(z)|.\]
By \rf{E:ball-measure} we have $\mu(B(f(x);s))\eqx_C\mu(B(f(x);t))$, for some $C$ depending only on the doubling constant for $\mu$ and the bi-Lipschitz constant for $f$. Furthermore, by a lemma of Semmes (\cite[Lemma B.4.7, p.420]{Gromov-structures}) we obtain constants $0<\alpha$ and $1\leq D<+\infty$ independent of $s,t$ and $x$ such that 
\[D^{-1}\left(\frac{t}{s}\right)^\alpha\leq\frac{\mu(B(f(x);t))}{\mu(B(f(x);s))}\leq C.\]
In particular, $t\leq (CD)^{1/\alpha}s$, and so $f:\mfR^2\to\mfR^2$ is quasisymmetric.

Using the quasisymmetry of $f$ and the doubling property of $\mu$ it is straightforward to verify that
\begin{equation}\label{E:image-measure}
\mu(B(f(x);|f(x)-f(y)|))\eqx \mu(f(B(x;|x-y|))).
\end{equation}
It then follows from the definition of $\mu$, the fact that $f:(\mfR^2,d_\mu)\to(\mfR^2,d_\mu)$ is bi-Lipschitz, \rf{E:image-measure}, and the quasiconformality of $f:\mfR^2\to\mfR^2$ that
\begin{align*}
\int_{B(x;|x-y|)}\omega(z)d\mathcal{H}^2(z)&\eqx\int_{f(B(x;|x-y|))}\omega(z)d\mathcal{H}^2(z)\\
&=\int_{B(x;|x-y|)}\omega(f(z))Jf(z)d\mathcal{H}^2(z).
\end{align*}
Since this holds for all $x,y\in\mfR^2$, we conclude that for almost every $z\in\mfR^2$ we have $\omega(z)\eqx\omega(f(z))Jf(z)$, as desired.

\medskip
Conversely, assume that $f:\mfR^2\to\mfR^2$ is a quasiconformal homeomorphism such that for almost every $x\in\mfR^2$ we have $\omega(x)\eqx \omega(f(x))Jf(x)$. For points $x,y\in\mfR^2$, by the quasisymmetry of $f$ and the doubling property of $\mu$ we have
\[\delta_\mu(f(x),f(y))=\mu(B_{f(x)f(y)})^{1/2}\eqx\mu(f(B_{xy}))^{1/2}.\]
Then using our assumption we find that
\begin{align*}
\mu(f(B_{xy}))=\int_{f(B_{xy})}\omega(z)d\mathcal{H}^2(z)&=\int_{B_{xy}}\omega(f(z))Jf(z)d\mathcal{H}^2(z)\\
&\eqx\int_{B_{xy}}\omega(z)d\mathcal{H}^2(z)=\delta_\mu(x,y)^2.
\end{align*}
It follows that $d_\mu(f(x),f(y))\eqx d_\mu(x,y)$, as desired.
\end{proof}

\section{Quasihomogeneous and bi-Lipschitz maps}\label{S:QH}

Recall the definition of a quasisymmetric embedding $f:X\to Y$. There exists a homeomorphism $\eta:[0,+\infty)\to[0,+\infty)$ such that for all triples of distinct points $x,y,z\in X$ we have
\[\frac{d_Y(f(x),f(y))}{d_Y(f(x),f(z))}\leq\eta\left(\frac{d_X(x,y)}{d_X(x,z)}\right).\]

A \textit{quasihomogeneous} embedding $h:X\to Y$ is a special case of the above: There exists a homeomorphism $\eta:[0,+\infty)\to[0,+\infty)$ such that for every collection of four distinct points $x,y,z,w$ we have
\[\frac{d_Y(h(x),h(y))}{d_Y(h(z),h(w))}\leq \eta\left(\frac{d_X(x,y)}{d_X(z,w)}\right).\]
Such maps are natural to work with in the setting of bi-Lipschitz homogeneous spaces. For example, it is easy to check that if $X$ is $L$-bi-Lipschitz homogeneous then $f(X)$ is $\eta(L)$-bi-Lipschitz homogeneous. We also refer the reader to results such as \cite[Theorem E]{HM-blh} and \cite[Corollary 1.2]{Bishop-blh}.

The strategy underlying the following proof was communicated to the author by David Herron. For other results about the relationship between quasihomogeneous and bi-Lipschitz mappings, see \cite[Section 1.4]{Aseev02}. Note also that it is stated (without proof) in \cite{Shag93} that any quasihomogeneous mapping between open subsets of $\mathbb{R}^n$ is bi-Lipschitz. 

\begin{proof}[Proof of \rf{T:QHBL}]
Suppose $(X,d)$ is an Ahlfors $n$-regular metric space and $f:\mfR^n\to (X,d)$ is a homeomorphism. If $f$ is bi-Lipschitz, then it is trivially quasihomogeneous. Therefore, we focus on the reverse implication. Assume that $f$ is $\eta$-quasihomogeneous. We claim:
\begin{equation}\label{E:cubes}
\text{For every cube } Q\subset \mfR^n \text{ we have } \mathcal{H}^n(f(Q))\eqx\mathcal{H}^n(Q).
\end{equation}
Here the comparability is independent of $Q$. Since $f$ is quasisymmetric and $X$ is Ahlfors $n$-regular, the measure $\mu(E):=\mathcal{H}^n(f(E))$ is doubling. Furthermore, $d_\mu(x,y)\eqx \delta_\mu(x,y)$, where $\delta_\mu(x,y)$ is the quasidistance associated to $\mu$ and $d_\mu(x,y):=d(f(x),f(y))$ (compare with \cite[(1.17)]{BHS-QC}). Thus $\mu$ is a metric doubling measure on $\mfR^n$ with strong $A_\infty$ weight denoted by $\omega$, and $f:(\mfR^n,d_\mu)\to(X,d)$ is an isometry.

Assuming that \rf{E:cubes} is true, for every cube $Q\subset\mfR^n$ we have
\[\int_Q\omega(z)d\mathcal{H}^n(z)=\int_{f(Q)}d\mathcal{H}^n(z)\eqx\int_Qd\mathcal{H}^n(z).\]
This is enough to guarantee the existence of $1\leq C<+\infty$ such that $\omega(z)\eqx_C1$ almost everywhere in $\mfR^n$, which tells us that $\id:\mfR^n\to(\mfR^n,d_\mu)$ is bi-Lipschitz. Since $f:(\mfR^n,d_\mu)\to(X,d)$ is an isometry, we conclude that $f:\mfR^n\to (X,d)$ is bi-Lipschitz, as desired. Thus it suffices to verify \rf{E:cubes}.

To this end, we note that a quasihomogeneous map is necessarily quasisymmetric. Therefore, if $E$ is an open subset of $\mfR^n$ such that 
\[B(x;r)\subset E\subset B(x;Cr)\]
for some $r>0$ and $1\leq C<+\infty$, then $\mathcal{H}^n(f(E))\eqx |f(x)-f(y)|^n$ for any  $y\in S(x;r)$, where the comparability depends only on $C$ and $\eta$.  Here we use the $n$-regularity of $X$. Equivalently, for such $E$ we have 
\begin{equation}\label{E:measure}
\mathcal{H}^n(f(E))\eqx\diam(f(E))^n.
\end{equation}

Suppose that there exists some sequence of cubes $Q_i\subset\mfR^n$ with sidelengths $2^{-m_i}\to0$ and a sequence of positive real numbers $C_i\to+\infty$ such that 
\begin{equation}\label{E:blow-up}
\frac{\diam(f(Q_i))}{\diam(Q_i)}>C_i.
\end{equation}
Using quasihomogeneity, for \textit{every} cube $Q\subset\mfR^n$ with sidelength $2^{-m_i}$ we will have $\diam(f(Q))/\diam(Q)>C_i/C_0$ where $C_0$ depends only on $\eta$.  

Given any cube $Q^0$ of sidelength $2^0=1$, for $i>1$ we have $Q^0=\cup_jQ_j^{-m_i}$, where the cubes $\lbrace Q^{-m_i}_j\rbrace$ form a dyadic partition of $Q^0$ such that each $Q^{-m_i}_j$ has sidelength $2^{-m_i}$. Since $\mathcal{H}^n(f(E))=\int_E\omega(z)d\mathcal{H}^n(z)$ it is clear that $f$ preserves sets of $\mathcal{H}^n$-measure zero. Thus we have $\mathcal{H}^n(f(Q^0))=\sum_j\mathcal{H}^n(f(Q^{-m_i}_j))$. Using this along with \rf{E:measure} and \rf{E:blow-up},
\begin{align*}
\mathcal{H}^n(f(Q^0))&=\sum_j\mathcal{H}^n(f(Q^{-m_i}_j))\\
&\gex (C_i/C_0)^n\sum_j\mathcal{H}^n(Q^{-m_i}_j)=(C_i/C_0)^n\mathcal{H}^n(Q^0).
\end{align*}
Note that $\mathcal{H}^n(f(Q^0))<+\infty$. Since the right hand side of this inequality tends to infinity as $i\to\infty$, we obtain a contradiction. Similar arguments also yield a contradiction if we have $\diam(f(Q_i))/\diam(Q_i)$ tending to 0. Therefore, there exists some $1\leq C<+\infty$ such that any cube of sidelength $2^{-m}$ (for $m\in \mfN\cup\{0\}$) satisfies
\[C^{-1}\mathcal{H}^n(Q)\leq\mathcal{H}^n(f(Q))\leq C\,\mathcal{H}^n(Q)\]
Now consider a dyadic cube $Q^n$ of side length $2^n$ for some $n\in\mfN$. Then 
\[\mathcal{H}^n(f(Q^n))=\sum_j\mathcal{H}^n(f(Q^0_j))\eqx\sum_j\mathcal{H}^n(Q^0_j)=\mathcal{H}^n(Q^n).\]
Here $\{Q^0_j\}_j$ forms a dyadic partition of $Q^n$ by cubes of side length $2^0=1$.

Using quasihomogeneity and \rf{E:measure} we can see that the Hausdorff $n$-measure of \textit{any} cube will quasi-preserved as in \rf{E:cubes}.
\end{proof}

We remark that quasihomogeneous maps need not be bi-Lipschitz in general. For example, any (unbounded) bi-Lipschitz homogeneous Jordan line $\Gamma\subset\mfR^2$ has a quasihomogeneous parameterization $f:\mfR\to\Gamma$, but if the Hausdorff dimension of $\Gamma$ exceeds 1 it does not have a bi-Lipschitz parameterization (see \cite{Freeman-blh}). For an interesting 2-dimensional example, consider the embedding $f:\mfR^2\to\mfR^3$ constructed in \cite[Theorem 1.1]{Bishop-qh-surface}.

\section{Bilipschitz Images of the Euclidean Plane}\label{S:BL}

\begin{proof}[Proof of \rf{T:R2}]
First we prove $(a)\Rightarrow(b)$. Given our assumptions on $X$, by \cite[Theorem 1.2]{Wildrick-planes} there exists a quasisymmetric homeomorphism $f:\mfR^2\to X$. For any measurable set $E\subset\mfR^2$, define $\mu(E):=\mathcal{H}^2(f(E))$. Then $\mu$ is a metric doubling measure on $\mfR^2$ with constant depending only the quasisymmetry function for $f$, and $\mu$ has an $A_\infty$-density $\omega$ such that $d\mu(x)=\omega(x)d\mathcal{H}^2(x)$ (so $\omega$ is a strong $A_\infty$ weight). Let $d_\mu$ denote a distance function comparable to the quasidistance $\delta_\mu$ associated to $\mu$ (we could take $d_\mu(x,y):=d(f(x),f(y))$). Then $f:(\mfR^2,d_\mu)\to(X,d)$ is a bi-Lipschitz homeomorphism. 

By assumption, there exists a group $G$ of uniformly bi-Lipschitz self homeomorphisms acting transitively on $X$. Then the group $G':=f^{-1}Gf$ is a group of uniformly bi-Lipschitz self homeomorphisms acting transitively on $(\mfR^2,d_\mu)$. Let $1\leq L<+\infty$ denote the uniform bi-Lipschitz constant for maps in $G'$. As in \rf{L:dbling-bl}, the group $G'$ can be viewed as a group of uniformly quasiconformal self homeomorphisms acting transitively on $\mfR^2$. By \cite[Theorem 12.1]{Gehring-topics} (and references therein) there exists a quasiconformal self homeomorphism $h:\overline{\mfR^2}\to\overline{\mfR^2}$ and a M\"obius group $M$ such that $G'=h^{-1}Mh$. Let $g=h^{-1}\phi h$ denote an element of $G'$. By \rf{L:dbling-bl}, for almost every $x\in\mfR^2$ we have
\[\omega(x)\eqx\omega(g(x))Jg(x)=\omega(g(x))J(h^{-1}\phi h)(x).\]
Using the chain rule, for almost every $x\in\mfR^2$ we have
\[J(h^{-1}\phi h)(x)=\frac{J\phi(h(x))Jh(x)}{Jh(g(x))}.\]
Putting these statements together we conclude that for almost every $x\in\mfR^2$,
\begin{equation}\label{E:compare1}
\omega(x)\eqx Jh(x)J\phi(h(x))\frac{\omega(g(x))}{Jh(g(x))}.
\end{equation}

Note that (up to conjugation by an additional M\"obius map of $\overline{\mfR^2}$) we may assume that $M$ fixes $\infty$. Under this assumption we claim that $M$ consists only of rotations and/or translations. 

To verify this claim, let $\phi\in M$ be given. Since $\phi$ fixes $\infty$, it can be written in the form $\phi(z)=az+b$ for $a,b\in \mfC$ such that $a\not=0$. We consider the behavior of the iterated functions 
\[\phi^n(z)=\overbrace{\phi\circ \phi\circ \dots\circ \phi}^\text{$n$ times}(z).\] 
Suppose $\phi(z)=az+b$. If $a=1$ and $b\not=0$, then for each $z\in \mfC$, the sequence $(\phi^n(z))$ tends to $\infty$. When $a\not=1$, the map $\phi$ has a unique fixed point in $\mfC$ at $z_\phi:=b/(1-a)$. If $|a|>1$ then for each $z\not =z_\phi$ we have $\phi^n(z)\to\infty$. If $|a|<1$ then for each $z\in\mfC$ we have $\phi^n(z)\to z_\phi$. If $|a|=1$, then $\phi$ is simply a rotation of the plane about the fixed point.

Suppose there exists a map $\phi(z)=az+b$ in $M$ with $|a|\not=1$. Then (up to taking the inverse of $\phi$) we may assume that $|a|<1$. Let $x$ and $y$ denote two points in the plane. Defining $x_n:=\phi^n(x)$ and $y_n:=\phi^n(y)$, both $x_n$ and $y_n$ tend to the fixed point $z_\phi$; so $|x_n-y_n|\to0$. However, since each $g^n:=h^{-1}\phi^nh$ is an $L$-bi-Lipschitz map of $(\mfR^2,d_\mu)$, for each $n$ we have $d_\mu(x_n,y_n)\eqx d_\mu(x,y)$. This contradicts the fact that $\mfR^2$ is homeomorphic to $(\mfR^2,d_\mu)$, and we conclude that
\begin{equation}\label{E:unit}
\forall \phi\in M, \quad \phi(z)=az+b \text{ for } a,b\in \mfC \text{ s.t. }|a|=1.
\end{equation}
With \rf{E:unit} in mind, we return to \rf{E:compare1} to obtain
\begin{equation}\label{E:compare2}
\omega(x)\eqx Jh(x)J\phi(h(x))\frac{\omega(g(x))}{Jh(g(x))}=Jh(x)\frac{\omega(g(x))}{Jh(g(x))}.
\end{equation}
We emphasize that the comparability is independent of the map $g$.

We now conclude this portion of the proof using the transitivity of $G'$. Fix a point $x\in \mfR^2$ at which \rf{E:compare2} holds. For any other point $z\in\mfR^2$ at which \rf{E:compare2} holds, there exists a map $g_z\in G'$ such that $g_z(x)=z$. Since the comparability in \rf{E:compare2} is independent of the map $g$ we find that $\omega(z)/Jh(z)\eqx\omega(x)/Jh(x)$ up to a constant independent of $z$. Since $z$ was any point in $\mfR^2$ at which \rf{E:compare2} holds, we conclude that $\omega(z)\eqx Jh(z)$ at almost every point of $\mfR^2$, up to a constant independent of $z$. This implies that $h:(\mfR^2,d_\mu)\to\mfR^2$ is bi-Lipschitz, and so $h\circ f^{-1}:(X,d)\to\mfR^2$ is a bi-Lipschitz homeomorphism.

\medskip
Since $(b)\Rightarrow(c)$ is trivial, we proceed to prove $(c)\Rightarrow(a)$. Let $h:\mfR^2\to X$ denote a quasihomogeneous homeomorphism. For any $x,y\in X$, let $\phi$ denote a translation of $\mfR^2$ taking $h^{-1}(x)$ to $h^{-1}(y)$. Then $g:=h\circ \phi\circ h^{-1}:X\to X$ takes $x$ to $y$. To see that $g$ is bi-Lipschitz, let $z,w\in X$. By quasihomogeneity,  
\[d(h(\phi(h^{-1}(z))),h(\phi(h^{-1}(w))))\eqx d(h(h^{-1}(z),h(h^{-1}(w)))=d(z,w)\]
Note that we do not need to use the 2-regularity of $X$ in this implication.
\end{proof}

\section{Inversion Invariant Bilipschitz Homogeneity}\label{S:inversion}

In \cite{BK-QMH}, Bonk and Kleiner generalized the notion of chordal distance on the Riemann sphere to unbounded locally compact metric spaces. In \cite{BHX-invert}, Buckley, Herron, and Xie built on this notion to develop the concept of \textit{metric space inversions.} We record a few pertinent facts about such inversions. Define $\hat{X}:=X\cup\{\infty\}$. Given a base point $p\in X$ and any two points $x,y\in X_p:=X\setminus\{p\}$ we define:
\[i_p(x,y):=\frac{d(x,y)}{d(x,p)d(y,p)}\quad \text{and} \quad i_p(x,\infty):=\frac{1}{d(x,p)}\]
This does not define a distance function in general, but one can show that 
\[d_p(x,y):=\inf\left\{\sum_{i=0}^{k-1}i_p(x_i,x_{i+1}):x=x_0,\dots,x_k=y\in X_p\right\}\]
defines a distance such that for all $x,y\in \hat{X}_p$, 
\[\frac{1}{4}i_p(x,y)\leq d_p(x,y)\leq i_p(x,y).\]

We use the distance $d_p$ to define the \textit{metric inversion of $X$ at $p$}, denoted by $X_p^*:=(\hat{X}_p,d_p)$. The identity map from $(\hat{X}_p,d)$ to $X^*_p$ is written as $\phi_p:\hat{X}_p\to X^*_p$. For points $x\in X_p$, it is often convenient to write $x^*:=\phi_p(x)$. We write $p^*$ to denote $\phi_p(\infty)$. So for $x\in X_p$ we have $1/4d(x,p)\leq d_p(x^*,p^*)\leq1/d(x,p)$. 

The following elementary estimate will be useful (see \cite[pg. 848]{BHX-invert}).

\begin{fact}\label{F:annulus-distortion}
For $0<r<R<\diam(X)$ and $x,y\in A_d(p;r,R)$, we have
\[\frac{d(x,y)}{4R^2}\leq d_p(\phi_p(x),\phi_p(y))\leq\frac{d(x,y)}{r^2},\]
\end{fact}

Given a metric space $X$, we use the term \textit{inversion invariant bi-Lipschitz homogeneity} to describe the situation in which both $X$ and $X^*_p$ are uniformly bi-Lipschitz homogeneous. The space $X$ is $L$-inversion invariant bi-Lipschitz homogeneous if the bi-Lipschitz maps in the definition are uniformly $L$-bi-Lipschitz. Note that this definition is independent (up to a controlled change in $L$) of the basepoint $p$ (see \cite[Lemma 3.2]{BHX-invert}).

\medskip
We say that a metric space $X=(X,d)$ is $L$-\textit{quasi-self-similar} provided that for some $x\in X$ and every $s\geq1$ there exists an $L$-bi-Lipschitz homeomorphism $f_s:(X,sd,x)\to(X,d,x)$. We emphasize that $L$ does not depend on $s$.

\begin{proof}[Proof of \rf{T:QSS}]
Let $1\leq L<+\infty$ denote the inversion invariant bi-Lipschitz homogeneity constant for $X$, and let $r>L^2$ be fixed. We claim that for any $x\in X$ and $s>0$ there exists an $M$-bi-Lipschitz embedding of $B_{rd}(x;s)$ into $(X,d)$ fixing the point $x$, where $M$ is determined only by $L$.

To verify this claim, begin by choosing $x\in (X,d)$ and $s>0$. Note that as sets, $B_d(x;s/r)=B_{rd}(x;s)$. Fix a point $p\in X$. By assumption, there exists an $L$-bi-Lipschitz homeomorphism $f:X\to X$ such that $y:=f(x)\in S_d(p;2s/\sqrt{r})$. Since $r>L^2$, we find that 
\[f(B_d(x;s/r))\subset A_d(p;s/\sqrt{r},4s/\sqrt{r})\]
Now we apply $\phi_p$. Due to \rf{F:annulus-distortion} we obtain
\[B_{d_p}(y^*;1/(64Ls))\subset\phi_p\circ f(B_d(x;s/r))\subset B_{d_p}(y^*;L/s).\]
Here $y^*:=\phi_p(y)\in X^*_p$. By assumption, there exists an $L$-bi-Lipschitz map $g:X_p^*\to X_p^*$ such that $z^*:=g(y^*)\in S_{d_p}(p^*;2L^2/s)$. Then we note that
\[B_{d_p}(z^*;L^2/s)\subset A_{d_p}(p^*;L^2/s,4L^2/s).\]
Let $Z:=\hat{X}_p$, so $\phi_p(\infty)=p^*\in(Z,d_p)$. Slightly abusing notation, we write $d_{p^*}$ to denote $(d_p)_{p^*}$, and we apply $\phi_{p^*}:(\hat{Z}_{p^*},d_p)\to(\hat{Z}_{p^*},d_{p^*})$. Writing $z:=\phi_{p^*}(z^*)$, by way of \rf{F:annulus-distortion} we have
\begin{align*}
B_{d_{p^*}}(z;s/(4096L^6))\subset\phi_{p^*}\circ g \circ \phi_p\circ f(B_d(x;s/r))\subset B_{d_{p^*}}(z;s/L^2).
\end{align*}

Now we examine the map $\Phi:=\id^*\circ \phi_{p^*}\circ g\circ \phi_p\circ f:(X,d)\to(X,d)$. Here $\id^*$ is the identity map from $(\hat{Z}_{p^*},d_{p^*})$ to $(X,d)$. By \cite[Proposition 3.3]{BHX-invert} the map $\id^*$ is $16$-bi-Lipschitz. The above results yield
\[B_d(z;s/(10^6L^6))\subset\Phi(B_{d}(x;s/r))\subset B_d(z;16s/L^2).\]
Moreover, we find that for $a,b\in B_{d}(x;s/r)$, we have
\begin{align*}
d(\Phi(a),\Phi(b))&\eqx_{16}d_{p^*}(\phi_{p^*}\circ g\circ\phi_p\circ f(a),\phi_{p^*}\circ g\circ\phi_p\circ f(b))\\
&\eqx_{64L^4}s^2\,d_{p}(g\circ \phi_p\circ f(a),g \circ \phi_p\circ f(b))\\
&\eqx_{L}s^2\,d_p(\phi_p\circ f(a),\phi_p\circ f(b))\\
&\eqx_{64}r\,d(f(a),f(b))\\
&\eqx_{L}r\,d(a,b)
\end{align*}
Since $B_d(x;s/r)=B_{rd}(x;s)$, we conclude that $\Phi:B_{rd}(x;s)\to (X,d)$ is an $M$-bi-Lipschitz embedding, with $M$ determined only by $L$. Via one more $L$-bi-Lipschitz self-homeomorphism of $X$ taking $\Phi(x)\mapsto x$, our claim is verified.

The above claim allows us to prove that $(X,rd)$ is bi-Lipschitz equivalent to $(X,d)$. To see this, let $n\in \mfN$. The above claim yields an $M$-bi-Lipschitz map $f_n:B_{rd}(x;n)\to (X,d)$ with $f_n(x)=x$. For any $k\geq n$, the map $f_k:B_{rd}(x;k)\to (X,d)$ is $M$-bi-Lipschitz on $B_{rd}(x;n)$ and fixes $x$. By the Arzel\`a-Ascoli Theorem, there exists a subsequence $(f_{k_l})_{l=1}^\infty$ that is locally uniformly convergent to an $M$-bi-Lipschitz embedding $F_n:B_{rd}(x;n)\to(X,d)$. By a diagonalization argument we obtain a subsequence that is locally uniformly convergent to an $M$-bi-Lipschitz homeomorphism $F_\infty:(X,rd)\to(X,d)$.
\end{proof}

The reader may notice that the proof of \rf{T:QSS} employs techniques very similar to those used in \cite{Freeman-iiblh}. In fact, we can use the above result to prove an analogue to \cite[Theorem 1.2]{Freeman-iiblh}. We remind the reader that a space is \textit{locally contractible} provided that every point $x\in X$ has a neighborhood that is contractible (i.e.\,a neighborhood on which the identity map is null-homotopic in $X$). A space is \textit{linearly locally contractible} provided there exists a constant $1\leq C<+\infty$ such that each ball $B(x;r)\subset X$ is contractible in $B(x;Cr)$. More explicitly, there exists a continuous map $H:B(x;r)\times[0,1]\to B(x;Cr)$ such that $H(\cdot\,,0)=\id$ and $H(\cdot\,,1)=x$.

\begin{lemma}\label{L:contractible}
Suppose $X$ is a proper, connected, locally contractible, and inversion invariant bi-Lipschitz homogeneous space. Then $X$ is linearly locally contractible. 
\end{lemma}

\begin{proof}
Let $x\in X$ be fixed. By assumption, there exists a contractible neighborhood $U\subset X$ that contains $x$. Choose $s>0$ such that $B_d(x;s)\subset U$. Since the identity map on $U$ is null-homotopic in $X$, its restriction to $B_d(x;s)$ is also null-homotopic in $X$. Therefore, there exists a constant $1\leq C<+\infty$ and a continuous map $H:B_d(x;s)\times[0,1]\to B_d(x;Cs)$ with $H(\cdot\,,0)=\id$ and $H(\cdot\,,1)=x$. For each $z\in B_d(x;s)$, write $H_t(z):=H(z,t)$. 

By \rf{T:QSS}, there exists a constant $1\leq M<+\infty$ such that $X$ is $M$-quasi-self-similar. By increasing $M$ if necessary, we can also assume that $X$ is $M$-bi-Lipschitz homogeneous. 

Now let $y\in X$ and $r>0$, and set $r_0:=s/(M^2r)$. By assumption, there is an $M$-bi-Lipschitz homeomorphism $g:X\to X$ with $g(y)=x$. Let $\id_0:(X,d)\to (X,r_0d)$ denote the identity map. By $M$-quasi-self-similarity there exists an $M$-bi-Lipschitz homeomorphism $f_0:(X,r_0d,x)\to (X,d,x)$. Define $F:=f_0\circ \id_0\circ g$ and we have $F(B_d(y;r))\subset B_d(x;s)$. Define $\tilde{H}:B_d(y;r)\times[0,1]\to X$ as $\tilde{H}(z,t):=F^{-1}\circ H(F(z),t)$ and we see that $\tilde{H}$ is continuous, $\tilde{H}_0$ is the identity map from $B_d(y;r)$ to itself, and $\tilde{H}_1$ is the constant map from $B_d(y;r)$ onto $y$. Furthermore, for $t\in[0,1]$ one can check that $\tilde{H}_t(B_d(y;r))\subset B_d(y;CM^4r)$. We conclude that $X$ is linearly locally contractible with constant $CM^4$.
\end{proof}

\section{Bilipschitz Images of Carnot Groups}\label{S:BLC}

For our purposes (following \cite{Berestovskii-1}), a \textit{Finsler manifold} refers to a differentiable manifold $M$ equipped with a continuous function $F:TM\to\mfR$ which yields a norm when restricted to a particular tangent space $T_pM$. A Riemannian manifold is thus a special case of a Finsler manifold. 

Let $\Delta$ denote a distribution on a Finsler manifold $M$ (i.e.\,a smooth section of $TM$), and let $\Delta_p\subset T_pM$ denote the distribution at a point $p\in M$. Let  $\mathcal{B}_p=\{X_1,X_2,\dots,X_m\}$ denote a (local) basis of vector fields for $\Delta$. Let $\Delta_p^{[i]}$ denote the span of all Lie brackets of order $\leq i$ of elements from $\mathcal{B}_p$. Note that we may have $\Delta_p^{[i]}=\Delta_p^{[i+1]}$. If there exists some $j$ for which $\Delta_p^{[j]}=T_pM$, then $\Delta$ is said to be \textit{bracket generating}. This property is sometimes referred to as \textit{H\"ormander's condition}.

A \textit{Carnot group} $\mathbb{G}$ of step $n$ is a connected, simply connected, nilpotent Lie group with stratified Lie algebra $\Lie(\mathbb{G})=V_1\oplus V_2\oplus\dots\oplus V_n$. The \textit{layers} $V_i$ have the property that, for $1\leq j\leq n-1$, $[V_j,V_1]=V_{j+1}$, where $[X,Y]$ denotes the Lie bracket. Here $V_n\not=\{0\}$ and for each $1\leq j\leq n$ we have $[V_j,V_{n}]=\{0\}$. We refer to $V_1$ as the \textit{horizontal layer} of $\Lie(\mathbb{G})$, and fix some inner product on $V_1$. By left-translation we extend $V_1$ to a left-invariant distribution $\Delta$ on $\mathbb{G}$ with a left-invariant inner product $\langle\cdot,\cdot\rangle$. 

Using the left-invariant norm $\|\cdot\|$ on $\Delta$ obtained from $\langle\cdot,\cdot\rangle$, we define the associated \textit{Carnot-Caratheodory} (or \textit{sub-Riemannian}) distance $d_{CC}$ on $\mathbb{G}$ as follows. Let $\gamma:[0,1]\to\mathbb{G}$ be an absolutely continuous path with endpoints $x=\gamma(0)$ and $y=\gamma(1)$ in $\mathbb{G}$. The path $\gamma$ is \textit{horizontal} provided that for almost every $t\in[0,1]$ we have $\dot\gamma(t)\in \Delta$.  The $d_{CC}$ length of a horizontal path $\gamma$ is 
\begin{equation}\label{E:cc-length}
\ell_{CC}(\gamma):=\int_0^1\|\dot\gamma(t)\|dt.
\end{equation}
We then define 
\begin{equation}\label{E:cc-dist}
d_{CC}(x,y):=\inf\{\ell_{CC}(\gamma):\,\gamma \text{ a horizontal path from } x \text{ to } y\}.
\end{equation}
By well known results of Chow and Rashevskii, $d_{CC}$ is indeed a (finite) geodesic distance on $\mathbb{G}$ due to the fact that the horizontal layer of $\Lie(\mathbb{G})$ yields a left-invariant distribution on $\mathbb{G}$ that is bracket generating.

In our proof below, we need a generalized version the Carnot-Carath\'eodory distance defined via \rf{E:cc-length} and \rf{E:cc-dist} in the context of homogeneous spaces $G/H$. A \textit{homogeneous space $G/H$} is the quotient of a Lie group $G$ by a compact subgroup $H<G$. We assume that $G/H$ is equipped with a $G$-invariant distance. 

Given such a space $G/H$, let $\Delta$ denote a $G$-invariant distribution on $G/H$. Similar to the Carnot group setting, we say that a curve $\gamma$ is \textit{horizontal} provided that $\dot\gamma(t)\in\Delta$ for almost every $t$. Fixing a $G$-invariant norm $F$ on $\Delta$, we define the $\textit{sub-Finsler length}$ of a horizontal curve $\gamma$ in $G/H$ as in \rf{E:cc-length} with $F(\cdot)$ in the place of $\|\cdot\|$. When $\Delta$ is bracket generating, the \textit{Finsler-Carnot-Caratheodory distance} (or \textit{sub-Finsler distance}) $d_{SF}$ obtained via sub-Finsler length as in \rf{E:cc-dist} is indeed finite on $G/H$ (see, for example, \cite{Berestovskii-2}). 
 
\medskip
We define a notion of a \textit{tangent space} to a metric space $X$ using the definition found in \cite{LeDonne-tangents} (see \cite{Herron-ptd-gh} for an alternate definition and a detailed treatment of this concept). Indeed, we say that a (pointed) metric space $(X_\infty,d_\infty,x_\infty)$ is a tangent to the (pointed) metric space $(X,d,x)$ if there exists a sequence of Hausdorff approximations
\[\{\phi_i:(X_\infty,d_\infty,x_\infty)\to(X,\lambda_i d,x)\}_{i\in\mfN},\]
for some $\lambda_i\to+\infty$. That is to say that for all $R\geq0$ and $\delta>0$ we have
\[\limsup_{i\to\infty}\{|\lambda_id(y_i,z_i)-d_\infty(y,z)|:y,z\in B_\infty(R)\}=0.\]
Here $y_i:=\phi_i(y)$, $z_i:=\phi_i(z)$, and $B_\infty(R):=B_{d_\infty}(x_\infty;R)$. In addition,
\[\limsup_{i\to\infty}\{\text{dist}(u;\phi_i(B_\infty(R+\delta))):u\in B_{\lambda_id}(x;R)\}=0.\]

\medskip
With the above definitions in mind we are able to proceed with the proof of \rf{T:Carnot}. As stated in \rf{S:main}, the proof consists of an application of the ideas found in \cite{LeDonne-tangents} and an application of \rf{T:QSS}. 

\begin{proof}[Proof of \rf{T:Carnot}]
Assume that $X$ is inversion invariant bi-Lipschitz homogeneous with respect to a group. Using this group structure we distort $X$ so that it becomes isometrically homogeneous. To do this, first define $d_G(x,y):=\sup_{g\in G}\{d(g(x),g(y))\}$. It is straightfoward to check that $d_G$ is a distance function. Since $G$ consists of uniformly $L$-bi-Lipschitz maps we have $d_G\eqx_L d$. Next, define
\[\ell_G(x,y):=\inf\{d_G\text{-length}(\gamma):\gamma \text{ is a path from } x \text{ to } y\}\] 
Here $d_G$-length is defined as
\[d_G\text{-length}(\gamma):=\sup_\mathcal{P}\left\{\sum_{i=1}^{n-1}d_G(x_i,x_{i+1})\right\},\]
where the supremum is taken over all partitions $\mathcal{P}=\{x_i\}_{i=1}^n$ of $\gamma$ consisting of ordered points along $\gamma$ such that  $x_1=x$ and $x_n=y$.
Since $(X,d)$ is a geodesic space, it follows that $\ell_G\eqx d$. Furthermore, since the distance $d_G$ is invariant under the action of $G$, so is $\ell_G$. Defining $Y:=(X,\ell_G)$ we obtain a geodesic space that is bi-Lipschitz equivalent to $X$ and on which $G$ acts as a transitive group of isometries.

As in \cite{LeDonne-tangents}, we apply results of Gleason, Montgomery, Zippin, and Yamabe to conclude that $G$ can be given the structure of a Lie group (see \cite[Theorem 3.6]{LeDonne-tangents}). We now explain how Le Donne's methods can be used to verify that $X$ is bi-Lipschitz equivalent to a Carnot group $\mathbb{G}$ equipped with a Carnot-Caratheodory distance associated to the horizontal layer of $\Lie(\mathbb{G})$. 

By \cite[Theorem 3]{Berestovskii-2} we conclude that $Y$ is isometric to a homogeneous space $G/H$ endowed with a sub-Finsler distance $d_{SF}$ arising from a $G$-invariant norm $F$ on a bracket generating $G$-invariant distribution $\Delta$. Here $H$ is the stabilizer of some point in $Y$ under the action of $G$. 

As in the proof of \cite[Theorem 1.4]{LeDonne-tangents} we apply Mitchell's Theorem (\cite[Theorem 1]{Mitchell-tangents}) to ensure the existence of a Carnot group $\mathbb{G}$ with a Carnot-Carath\'eodory distance $d_{CC}$ associated to the first layer of $\Lie(\mathbb{G})$ such that the tangent space at any point in $p\in (G/H,d_{SF})$ is bi-Lipschitz equivalent to $\mathbb{G}$, with bi-Lipschitz equivalence constant independent of $p$. 

We make a few remarks to justify this application of Mitchell's Theorem, which is proved in the context of Carnot-Carath\'eodory distances defined via a Riemannian inner product. Let $\Delta$ denote the $G$-invariant distribution on $G/H$ obtained via \cite[Theorem 3]{Berestovskii-2}, and let $F$ denote a $G$-invariant norm on $\Delta$ that is used to obtain the distance $d_{SF}$. Since $\Delta$ is finite dimensional, it follows that any other $G$-invariant norm on $\Delta$ will give rise to a sub-Finsler distance that is bi-Lipschitz equivalent to $d_{SF}$. Therefore, the space $(G/H,d_{SF})$ is bi-Lipschitz equivalent to $(G/H,d_{SR})$, where $d_{SR}$ is a sub-Riemannian distance obtained from a $G$-invariant norm on $\Delta$ that is defined via a $G$-invariant inner product. Mitchell's Theorem applies directly to the space $(G/H,d_{SR})$. Since $(G/H,d_{SF})$ is bi-Lipschitz equivalent to $(G/H,d_{SR})$, any tangent of $(G/H,d_{SF})$ is bi-Lipschitz equivalent to a corresponding tangent of $(G/H,d_{SR})$. 

Since $X$ is quasi-self-similar (by \rf{T:QSS}), it is straightforward to check that any tangent to $X$ is bi-Lipschitz equivalent to $X$ itself. Indeed, suppose the spaces $(X,\lambda_id,x)$ converge to the tangent space $(X_\infty,d_\infty,x_\infty)$. By quasi-self-similarity, for each $i$ there exists a bi-Lipschitz homeomorphism $f_i:(X,\lambda_id,x)\to(X,d,x)$. The sequence $(f_i)$ induces a bi-Lipschitz homeomorphism $f_\infty:(X_\infty,d_\infty,x_\infty)\to(X,d,x)$ (see \cite{Herron-ptd-gh}). 

To conclude the proof, we recall that $(X,d)$ is bi-Lipschitz equivalent to $(G/H,d_{SF})$, and so by combining the previous three paragraphs we find that $(X,d)$ itself is bi-Lipschitz equivalent to the Carnot group $\mathbb{G}$.
\end{proof}

\section{Proofs of the Remaining Results}\label{S:remains}

Before beginning the proof of \rf{C:Rn}, we state two facts which allow us to circumvent the requirement of a geodesic distance appearing the statement of \rf{T:Carnot}. The first fact is \cite[Theorem 1.1]{Freeman-iiblh}.

\begin{fact}\label{F:ARLLC}
Suppose $X$ is a proper, connected, and doubling metric space. If $X$ is inversion invariant bi-Lipschitz homogeneous, then $X$ is Ahlfors $Q$-regular. 
\end{fact}

The second fact comes from \cite[Theorem B.6]{Semmes-curves}. 

\begin{fact}\label{F:Semmes}
Suppose $X$ is an orientable topological manifold of dimension $n$. If $X$ is Ahlfors $n$-regular and linearly locally contractible, then $X$ is quasiconvex. 
\end{fact}

\begin{proof}[Proof of \rf{C:Rn}]
We prove necessity, as sufficiency is straightforward. Thus we assume that $X$ is inversion invariant bi-Lipschitz homogeneous with respect to a group. By \rf{F:ARLLC} and \rf{L:contractible}, we are able to apply \rf{F:Semmes} in order to conclude that $X$ is quasiconvex. Via the Hopf-Rinow Theorem, this implies that $X$ is bi-Lipschitz equivalent to a geodesic space. Applying \rf{T:Carnot}, we conclude that $X$ is bi-Lipschitz equivalent to a Carnot group $\mathbb{G}$ equipped with a left-invariant Carnot-Caratheodory distance associated with the first layer $V_1$ of the stratified algebra $\Lie(\mathbb{G})$.
By assumption, the Hausdorff and topological dimensions of $X$ agree. Therefore, $V_1=\Lie(\mathbb{G})$, and so $\mathbb{G}$ is abelian. Since any two left-invariant inner products on the left-invariant extension of $V_1$ yield bi-Lipschitz equivalent Riemannian distances, we conclude that $\mathbb{G}$ is bilipschitz equivalent to $\mfR^n$.
\end{proof}
 
Lastly, we provide the following example. As noted in the proof below, this example is not rectifiably connected. It would be interesting to know if there exists a similar example that is geodesic (or at least quasiconvex).

\begin{example}\label{X:rug}
There exists a surface in $\mfR^3$ that is quasi-self-similar and uniformly bi-Lipschitz homogeneous with respect to a group, yet fails to be inversion invariant bi-Lipschitz homogeneous.
\end{example}

\begin{proof}
Fix $1<Q<2$ and let $\Gamma\subset\mfR^2$ denote a bi-Lipschitz homogeneous, Ahlfors $Q$-regular, and proper curve homeomorphic to $\mfR$. By the constructions in \cite{Freeman-blh}, such a curve exists (indeed, such a curve can be obtained as the Hausdorff limit of an explicitly defined sequence of piecewise-linear curves). 

By \cite[Theorem 1.1]{Freeman-blh} and \cite[Theorem E]{HM-blh}, there exists an $\eta$-quasi-homogeneous parameterization $f:\mfR\to\Gamma$ such that $|f(x)-f(y)|^{1/Q}\eqx|x-y|$. It follows as in the proof of \rf{T:R2} $(c)\Rightarrow(a)$ that the curve $\Gamma$ is uniformly bi-Lipschitz homogeneous with respect to a group. Since products of bi-Lipschitz maps are bi-Lipschitz, $S$ is also bi-Lipschitz homogeneous with respect to a group.

To see that $S$ is quasi-self-similar, we proceed as follows. Given $r>0$, define $f_r:\Gamma\to\Gamma$ as
\[f_r(x):=f(r^{1/Q}f^{-1}(x)).\]
Therefore, for two points $x,y\in\Gamma$ we have
\[|f_r(x)-f_r(y)|\eqx r|f^{-1}(x)-f^{-1}(y)|^Q\eqx r|x-y|\]
Then define $F_r:\Gamma\times\mfR\to\Gamma\times\mfR$ as $F_r(x,t):=(f_r(x),rt)$. We find that $F_r$ fixes the origin (which we may assume lies in $S$) and furthermore, 
\[d(F_r(x,t),F_r(y,s))=|f_r(x)-f_r(y)|+|rt-rs|\eqx r\,d((x,t),(y,s)).\]

Finally, we show that $S$ fails to be inversion invariant bi-Lipschitz homogeneous. This is due to the fact that $S$ is not rectifiably connected. Indeed, we note that through each point  $(x,t)\in S=\Gamma\times \mfR$ there exists precisely one locally rectifiable curve, namely $\{x\}\times\mfR$. Since inversion is locally bi-Lipschitz away from the point of inversion, there is precisely one locally rectifiable curve through each point of $S^*_0\setminus\{0\}$. Here $S^*_0$ denotes the Euclidean inversion of $\hat{S}$ at the origin. However, there are infinitely many rectifiable paths through the origin in $S^*_0$. To see this, note that the inversion of each line $\{x\}\times\mfR\subset S$ is locally rectifiable and passes through the origin in $S^*_0$. Since bi-Lipschitz maps preserve locally rectifiable paths, $S^*_0$ is not bi-Lipschitz homogeneous. 
\end{proof}

\bibliographystyle{amsalpha}  
\bibliography{4872}            

\end{document}